\documentclass{amsart}
\usepackage{amsmath,amsthm,amsfonts,amssymb,amscd}
\usepackage[all]{xy}
\usepackage[english,francais]{babel}
\usepackage{graphicx,color}
\usepackage[T1]{fontenc}
\usepackage[latin1]{inputenc}
\usepackage{ae}

\usepackage[mathscr]{eucal}
\usepackage{amsmath,amssymb}
\usepackage{graphics}
\usepackage{multirow}
\usepackage{hyperref}
\usepackage[active]{srcltx}

\newtheorem{thm}{Theorem}[section]
\newtheorem{THM}{Theorem}

\newtheorem{cor}[thm]{Corollary}

\newtheorem{lemma}[thm]{Lemma}

\theoremstyle{definition}

\newtheorem{remark}[thm]{Remark}

\newcommand{\nc}{\newcommand}
\nc {\hh}{\check{h}}
\nc {\DD}{\mathcal{D}}
\nc {\RR}{\mathcal{R}}
\nc {\Pp}{\mathbb{P}}
\nc {\Ss}{\mathcal{S}}
\nc {\PP}{\mathbb{P}^{2}}
\nc {\Pd}{ \check{\mathbb{P}}^{2}}
\nc {\WW}{\mathcal{W}}
\nc {\Sym}{\mathrm{Sym}}
\nc {\OO}{\mathcal{O}}
\nc {\CC}{\mathbb{C}}
\nc {\EE}{\mathcal{E}}
\nc {\MM}{\mathcal{M}}
\nc {\KK}{\mathcal{K}}
\nc {\PW}{\mathcal{P}}
\nc {\NW}{\mathcal{N}_{\WW}}
\nc {\FF}{\mathcal{F}}
\nc {\GG}{\mathcal{G}}
\nc {\ZZ}{\mathcal{Z}}
\nc {\LL}{\mathcal{L}}
\nc {\HH}{\mathcal{H}}
\nc {\NN}{\mathcal{N}}
\nc {\VV}{\mathcal{V}}
\nc {\Ww}{\mathbb{W}}
\nc {\QQ}{\mathbb{Q}}
\nc {\II}{\mathcal{I}}

\author{M. Falla Luza}
\email{maycolfl@impa.br}
\address{UFF, Universidad Federal Fluminense, rua M\'ario Santos Braga S/N, Niter\'oi, RJ-Brasil}
\author{F. Loray}
\email{frank.loray@univ-rennes1.fr}
\address{IRMAR, Universit\'e de Rennes 1, 35042 Rennes Cedex, France}
\date{}

\begin{document}
\selectlanguage{english}

\title[On the number of fibrations]{On the number of fibrations transversal to a rational curve in a complex surface}

\maketitle

\begin{abstract}
\selectlanguage{english}
We investigate the existence, and lack of unicity, of a holomorphic fibration 
by discs transversal to a rational curve in a complex surface.  

\smallskip

\selectlanguage{francais}
\noindent{\bf R\'esum\'e.} \vskip 0.5\baselineskip \noindent
{\bf Sur le nombre de fibrations transverses \`a une courbe rationnelle dans une surface. } \newline
Nous \'etudions l'existence et le d\'efaut d'unicit\'e de fibrations holomorphes en disques 
transverse \`a une courbe rationnelle dans une surface complexe.
\end{abstract}

\selectlanguage{english}

\section{Introduction}
Soit $U$ une surface complexe contenant une courbe rationnelle lisse $C$. 
On s'int\'eresse \`a la structure du germe de voisinage $(U,C)$.
Lorsque $C^2\le 0$, il est bien connu que $(U,C)$ est holomorphiquement \'equivalent
au voisinage de la section nulle dans l'espace total du fibr\'e normal $N_C$ (see \cite{Grauert,Savelev}) ;
on dit alors que $(U,C)$ est {\it lin\'earisable}. On d\'eduit ais\'ement
l'existence de nombreuses fibrations holomorphes par disques transverse \`a $C$ dans ce cas.
D'un autre c\^ot\'e, lorsque $C^2>0$, il y a un gros espace de module de germes de tels voisinages
$(U,C)$ modulo isomorphismes
(voir \cite{Hur,Mish}).
A contrario, nous montrons qu'il y a tr\`es peu de fibrations transverses (en g\'en\'eral aucune)  
dans ce cas.
Il y a des familles de dimension infinie de voisinages sans (resp. avec une unique) 
fibration pour chaque $C^2>0$. Lorsque $C^2=+1$, il existe aussi des familles de dimension
infinie de voisinages avec exactement deux fibrations. Notre r\'esultat principal est le suivant.

\begin{THM}\label{Principal} 
Soit $(U,C)$ un germe de surface au voisinage d'une courbe rationnelle $C$ (tout est lisse) avec auto-intersection $C^2>0$.
\begin{itemize}
\item Si $(U,C)$ admet au moins $3$ fibrations holomorphes transverses \`a $C$, alors $C^2=1$ et $(U,C)$ est lin\'earisable, i.e. holomorphiquement \'equivalent au voisinage d'une droite dans $\mathbb P^2$.
\item Si $C^2>1$ et $(U,C)$ admet au moins $2$ fibrations holomorphes transverses \`a $C$, 
alors $C^2=2$ et $(U,C)$ est holomorphiquement \'equivalent au voisinage de la diagonale dans $\mathbb P^1\times\mathbb P^1$.
\end{itemize}
\end{THM}

Un analogue non lin\'eaire de la dualit\'e projective entre les points et les droites de $\mathbb P^2$ 
\'etablie par Le Brun (see \cite[\S 1.3, 1.4]{LB}) nous fournit une correspondance bi-univoque
entre les germes de voisinages $(U,C)$ avec $C^2=1$ et les germes de structures projectives 
en $(\mathbb C^2,0)$ modulo $\mathrm{Diff}(\mathbb C^2,0)$ (tout est holomorphe).
Par une structure projective, on entend une collection de g\'eod\'esiques (une courbe passant 
par chaque point et chaque direction)
d\'efinie par une connexion affine (i.e. une connexion lin\'eaire sur le fibr\'e tangent).
De ce point de vue, on a une correspondance bi-univoque entre les fibrations transverses \`a $C$
et les d\'ecompositions de la famille de g\'eod\'esiques comme pinceau de feuilletages 
(voir par exemple \cite{Wo}),
ou encore de connexions affines \`a courbure nulle d\'efinissant la structure. 

Tous les r\'esultats pr\'esent\'es dans cette note seront d\'etaill\'es dans \cite{Nous}.


\section{Introduction}
Let $U$ be a smooth complex surface containing a smooth rational curve $C$. 
We want to understand the structure of the germ of neighborhood $(U,C)$.
When $C^2\le 0$, it is known that $(U,C)$ is holomorphically equivalent
to the neighborhood of the zero section in the normal bundle $N_C$ (see \cite{Grauert,Savelev}).
In this case, we say that $(U,C)$ is {\it linearizable}, and
one can easily deduce the existence of many fibrations by discs transversal to $C$.
On the other hand, when $C^2>0$, there is a huge moduli space of germs of neighborhoods $(U,C)$
(see \cite{Hur,Mish}).
However, we prove that there are very few (in general there are not) transversal fibrations in this latter case.
There are infinite dimensional families of neighborhoods without (resp. with a unique) 
fibration for any $C^2>0$. When $C^2=+1$, there also exist infinite dimensional families of neighborhoods
with $2$ fibrations. Our main result is

\begin{thm}\label{Principal} 
Let $(U,C)$ be a germ of surface neighborhood of a rational curve $C$ with self-intersection $C^2>0$, everything being smooth.
\begin{itemize}
\item If $(U,C)$ admits at least $3$ distinct fibrations by discs transversal to $C$, then $C^2=1$ and $(U,C)$ is  linearizable, i.e. holomorphically equivalent to the neighborhood of a line in $\mathbb P^2$.
\item If $C^2>1$ and $(U,C)$ admits at least $2$ distinct fibrations by discs transversal to $C$, then 
$C^2=2$ and $(U,C)$ is holomorphically equivalent to the neighborhood of the diagonal in $\mathbb P^1\times\mathbb P^1$.
\end{itemize}
\end{thm}

A non linear analogue of the projective duality between lines and points in $\mathbb P^2$ 
established by Le Brun (see \cite[\S 1.3, 1.4]{LB}) provides a one-to-one correspondance 
between germs of neighborhoods $(U,C)$ with $C^2=1$ and germs of projective structure 
at $(\mathbb C^2,0)$ up to $\mathrm{Diff}(\mathbb C^2,0)$ (everything is holomorphic).
By a projective structure, we mean a collection of geodesics (one curve for each point+direction)
defined by an affine connection (i.e. a linear connection on the tangent bundle).
From this point of view, there is a one-to-one correspondance between fibrations transversal to $C$
and decomposition of the projective structure as a pencil of foliations (see for instance \cite{Wo}),
or equivalently affine connection with vanishing curvature. 


\section{Normal Form}\label{sec:NormalForm}
Let us fix a coordinate $x:C\stackrel{\sim}{\to}\mathbb C\cup\{\infty\}$ and decompose it 
as $C=V_0\cup V_\infty$ where $V_i$ are disks around $x=i$ with $i=0, \infty$ overlapping on a neighborhood of the circle
$\{\vert x\vert=1\}$.
It is easy to see that a germ of surface $(U,C)$ can always be obtained by gluing two open sets $U_0 = V_0 \times \mathbb{D}_{\epsilon}$ and $U_\infty=V_{\infty} \times \mathbb{D}_{\epsilon}$, with coordinates $(x_i,y_i)$ by some analytic diffeomorphism of the form
$$
(x_{\infty}, y_{\infty})=\Phi(x_0, y_0)=\left( x_0^{-1} + \sum_{n\geq 1} a_{n}(x_0) y_{0}^{n}, \sum_{n\geq 1} b_{n}(x_0) y_{0}^{n} \right).
$$
which we shall call the \textit{cocycle} of the germ $(U,C)$. In restriction to $C$, we have $x=x_0=1/x_\infty$.
Of course different cocycles could give rise to isomorphic germs of surface. In this sense, it is shown in \cite{Mish} that 
we can always arrive to an almost unique normal form. In the special case $C^2=1$,
the statement can be settled as follows.

\begin{thm}\label{normal-formal-form}
Let $(U,C)$ be a germ of surface with $C^2=1$. Then, we can choose the corresponding cocycle 
in the following \emph{normal form}
$$
\Phi  = \left( \frac{1}{x} + \sum_{n\geq 4} (\sum_{k=3}^{n-1}\frac{a_{k,n}}{x^k} ) y^n   ,  \frac{y}{x} + \sum_{n\geq 3} (\sum_{k=2}^{n-1}\frac{b_{k,n}}{x^k})y^n\right).
$$
Moreover, $\Phi^i =(x+ \sum a_n^i(x)y^n, \sum b_n^i(x) y^n)$, $i=0, \infty$, are such that $\Phi^{\infty} \circ \Phi \circ \Phi^{0}$ is still in normal form (with possibly different coefficients) if, and only if, there are constants $\alpha, \beta, \gamma \in \mathbb{C}$ and $\theta \in \mathbb{C}^*$ such that

\begin{align*}
a_1^0 &= \theta (\alpha x+\beta), a_2^0 = \alpha^2 \theta^2 x + \gamma ,  a_3^0 =  \alpha^3 \theta^3 x + (\gamma \alpha \theta + \alpha \theta^3 b_{2,3}), \\
a_n^0 &= \alpha^n \theta^n x + \left( \gamma(\theta \alpha)^{n-2} + \theta^n \left(\sum_{k=1}^{n-2} {n-2 \choose k-1} \alpha^k b_{2,n-k+1}  \right) \right) , n \geq 4, \\
a_1^{\infty}&= \beta x+ \alpha , a_2^{\infty}=\beta^2 x+\frac{\gamma}{\theta^2},\\
a_n^{\infty}&= \beta^n x + \left( (n-1)\frac{\beta^{n-2}\gamma}{\theta^2} - (n-2)\alpha \beta^{n-1} - \sum_{k=1}^{n-2} k \beta^k b_{n-k, n-k+1} \right) , n\geq3, \\
b_n^0 &= \theta^n \alpha^{n-1}, \hspace{0.2cm} b_n^{\infty} = \frac{\beta^{n-1}}{\theta}, n\geq 1.
\end{align*}
\end{thm}

A very similar result can be found in \cite{Mish}; the proof of this statement 
(which will be needed to prove the main result) will be detailled in \cite{Nous}.
Let us just mention that normalizing coordinates $(x_0,y_0)$ (resp. $(x_\infty,y_\infty)$) 
are obtained after blowing-up $\infty\in C$ (resp. $0\in C$) by the identification
of the new germ of neighborhood (with $C^2=0$) with the product $C\times \mathbb{D}_{\epsilon}$.
The normal form follows from a good choice of these trivialisations.

\begin{remark}\label{normal-form-with-fibration}
When in addition $U$ admits a fibration transverse to $C$, then there exists a normal form compatible 
with the fibation in the sense that  $a_{kn}=0$ for every $k,n$ (with notation of Theorem \ref{normal-formal-form}) and the fibration is given by $\{x=\text{constant}\}$.
\end{remark}

\begin{remark}\label{case_k>1}
The classification of germs $(U,C)$ with $C^2=k>1$ reduces to the case $C^2=1$
since the unique cyclic $k$-fold cover $\pi:\tilde U\to U$ ramifying over $C$ is such that 
$\tilde C^2=1$, where $\tilde C=\pi^{-1}(C)$. 
\end{remark}

\begin{remark}\label{U(4)}
When $C^2=1$, the formal neighborhood $U^{(4)}$ always admits a transverse fibration: the cocycle can always 
be normalized to $(\frac{1}{x},\frac{y}{x}+\cdots)$ up to order $4$. The first 
obstruction for having a transverse fibration appears in order five. 
\end{remark}

{\bf Many examples without transverse fibration.}
Consider the neighborhood $U$ given by the cocycle
$$
\Phi = \left( \frac{1}{x}+ \frac{y^5}{x^3} + \sum_{n\geq 6} (\sum_{k=3}^{n-1}\frac{a_{k,n}}{x^k} ) y^n , \frac{y}{x}  \right)
$$
which is already in normal form and suppose that $U$ admits a transverse fibration. Then by remark \ref{normal-form-with-fibration} there is a normal form compatible with the fibration, that is, there exists $A=(\alpha, \beta, \gamma, \theta) \in \mathbb{C}^3 \times \mathbb{C}^*$ such that $a_{i,j}^A=0$ for every $i,j$, where the notation stands for the coefficients of the normal form after composing with the pair $(\Phi^0, \Phi^{\infty})$ associated to $A$ as in Theorem \ref{normal-formal-form}. On the other hand, we can compute $a_{3,4}^A =  -( \gamma - \alpha \beta \theta^2 )^2$ and conclude that $\gamma = \alpha \beta \theta^2$. This leads us to $a_{3,5}^A = \theta^5 \neq 0$, which is impossible.\\


{\bf Many examples with exactly one transverse fibration.}
Consider the neighborhood $U$ given by the cocycle
$$
\Phi = \left( \frac{1}{x}, \frac{y}{x} + \sum_{n\geq 5} (\sum_{k=2}^{n-1}\frac{b_{k,n}}{x^k})y^n\right).
$$
with $b_{3,5}^2 - b_{2,5}\cdot b_{4,5} \neq 0$, which is in normal form and has a compatible transverse fibration. If there exists another fibration then we can find some $A=(\alpha, \beta, \gamma, \theta) \in \mathbb{C}^3 \times \mathbb{C}^*$ such that $a_{i,j}^A=0$ for every $i,j$. We can also assume that $\theta=1$. We compute $a_{3,4}^A=-(\gamma -\alpha \beta)^2$, which implies $\gamma = \alpha \beta$. Thus $a_{3,5}^A = \alpha b_{3,5} + \beta b_{2,5}=0$ and $a_{4,5}^A = \alpha b_{4,5} + \beta b_{3,5}=0$. By hypothesis the only solution of this system is $\alpha = \beta=0$ that also gives $\gamma=0$. We conclude that $A=(0,0,0,1)$ and the fibration is the same that the initial one.


\section{Examples with 2 fibrations transverse to $C$}\label{tangencia en C}

Consider $(U, C)$ a germ of surface as before and suppose that there are two fibrations $\mathcal{G}$ and $\mathcal{G}'$ transverse to $C$. If they are not tangent along $C$ it is easy to see that their tangency locus is a (smooth) curve transversal to $C$ at one point.

\begin{thm}\label{tangency-along-C}
Let $U$ be a neighborhood of a rational curve $C$ with self intersection $+1$ and suppose that there are two fibrations $\GG$ and $\GG'$ transverse to $C$ such that $tang(\GG, \GG')=C$. Then $(U,C)$ is the ramified covering of degree $2$ over a neighborhood of the diagonal curve $\Delta \subseteq V \subseteq \Pp^1 \times \Pp^1$  with ramification locus $\Delta$ and branching locus $C$. Moreover, there is no other fibration on $U$ transversal to $C$.
\end{thm}

\begin{proof}[Idea of proof] Consider the first integrals $g,g':U\to C\simeq\mathbb P^1$ defined by 
the two fibrations; then the map $(g,g'):U\to \Pp^1 \times \Pp^1$ is the $2$-fold cover.
See \cite{Nous} for details.
\end{proof}

\begin{remark}
It is not hard to use \cite[Proposition 4.7]{Hur} or \cite[Proposition V.4.3]{BPV} in order to see that the germ $(U, C)$ of the theorem is not algebrizable. However, the field of meromorphic functions on $U$
identifies with $\mathbb C(x,y)$ and has therefore transcendence degree of $2$ over $\mathbb C$. 
\end{remark}

Suppose now that  $T = tang(\GG, \GG')$ is neither a common fiber nor $C$  and is transverse to $\GG$ and $\GG'$ at $C$. Take first integrals $g,g': U \rightarrow \Pp^1$ such that $g|_{C} \equiv g'|_{C} =x$, where $x$ is the global coordinate of $\Pp^1$, and define 
$$
\Phi : U \rightarrow \Pp^1 \times \Pp^1,  \hspace{0.1cm} p\mapsto (g(p), g'(p)).
$$
Observe that $\Phi|_{C}: C  \rightarrow \Delta$ is an isomorphism,  where $\Delta$ is the diagonal curve, and that $\Phi$ is a local biholomorphism outside $T$. It is possible to show that $\Phi(T)$ has a simple tangency with $\Delta$ in one point and that $\Phi$ is a $2:1$ covering ramifying over $\Phi(T)$. Changing the first integrals $(f,g)$ coinciding along $C$ is the same that considering the diagonal action of $PSL_2(\mathbb{C})$ on $(\Pp^1 \times \Pp^1, \Delta)$, that is, $\varphi \mapsto (\varphi(x), \varphi(y))$. Moreover, every curve with a simple tangency with $\Delta$ can be constructed by this process.

\begin{thm}
The moduli space of neighborhoods $(U,C)$ having two fibrations with a tangency locus $T$ which is neither a common fiber nor $C$,  and which is transversal to $\GG$ and $\GG'$ at $C$, is in bijection with the space of germs of curves on $(\Pp^1 \times \Pp^1, \Delta)$ with a simple tangency with $\Delta$ modulo the diagonal action of $PSL_2(\mathbb{C})$.
\end{thm}

\begin{cor}
The moduli space of neighborhoods $(U,C)$ having two fibrations transverse to $C$ is infinitely dimensional.
\end{cor}

\section{The case of $3$ transverse fibrations}
 Our goal is to show that the only case having $3$ fibrations is the case of the projective plane.
 
 \begin{thm}\label{3-fibrations}
Let $(U, C)$ be a germ of surface containing a rational curve $C$ with $C^2= +1$. Assume that $U$ admits $3$ regular fibrations $\GG$, $\GG'$ and $\GG''$ transverse to $C$, then $(U, C)$ is isomorphic to the germ $(\PP, L_0)$, where $L_0$ is a line in $\PP$.
 \end{thm}

The detailed proof will be given in \cite{Nous}, however, we present here an special case in order to show the main ideas.  Before proving the theorem we make some considerations. First of all, observe that there are normal forms 
$$
\Phi = \left(\frac{1}{x}, \frac{y}{x} + (\frac{b_{2,3}}{x^2})y^3 + \ldots\right),  \Phi^i = \left(\frac{1}{x}, \frac{y}{x} + (\frac{b^{i}_{2,3}}{x^2})y^3 + \ldots \right),\hspace{0.2cm} i=1,2,
$$
compatible with $\GG$, $\GG'$ and $\GG''$ respectively, that is, in which the fibration is given by $\{x=const.\}$ (Remark \ref{normal-form-with-fibration}). Denote by $(\alpha_i, \beta_i, \gamma_i, \theta_i)$ the parameter corresponding to the pair of biholomorphisms $(\Phi^0_i, \Phi^{\infty}_i)$ leaving $\Phi$ into $\Phi^i$ ($i=1,2$) given by Theorem \ref{normal-formal-form}. Note that coefficients $(a^i_{j,k}, b^i_{j,k})$ of the normal form $\Phi^i$ depend of the coefficients $b_{j,k}$ and $(\alpha_i, \beta_i, \gamma_i, \theta_i)$. We clearly have $a^i_{j,k}=0$ for $i=1,2$.

By composing $\Phi$ with the change of coordinates associated to $(0,0,0,\theta)$, which does not affect the fibrations, we can assume that $b_{2,3}=1$ or $b_{2,3}=0$. In a similar way we can also suppose that $\theta_1 = \theta_2 =1$.

\xymatrix{
&&&&& \Phi  \ar[ld]_{(\alpha_1, \beta_1, \gamma_1, 1 )} \ar[rd]^{(\alpha_2, \beta_2, \gamma_2, 1 )} & \\
&&&&\Phi^1  & & \Phi^2 
}

We will need the following lemma.

\begin{lemma}\label{(a1,0,0,1), (0,b2,0,1)}
If $b_{2,3}=0$, $\alpha_1 \neq 0$, $\beta_1=\gamma_1=0$, and $\beta_2\neq 0$, $ \alpha_2=\gamma_2=0$, then $b_{ij}=0$ for every $i,j$.
\end{lemma}

\begin{proof}[Proof of Theorem \ref{3-fibrations}]
We consider the case \textbf{$tang(\GG, \GG') \cap C  \neq  tang(\GG, \GG'') \cap C $}. 

Thus we assume $tang(\GG, \GG') \cap C =\{0\}$ and $tang(\GG, \GG'') \cap C =\{\infty \}$. This implies, by the form of the pair $(\Phi^0_i, \Phi^{\infty}_i)$ associated to $(\alpha_i, \beta_i, \gamma_i, 1)$ given in Theorem \ref{normal-formal-form}, that $\beta_1 = \alpha_2 =0$. Observe also that if $\alpha_1=0$ then $\GG$ and $\GG'$ are tangent along $C$, but Theorem \ref{tangency-along-C} implies that this is not the case, thus we can suppose $\alpha_1 \neq 0$. Analogously we also assume $\beta_2 \neq 0$. 

If $b_{2,3}=1$ we consider the equations 
\begin{eqnarray*}
a^1_{3,4} = a^1_{3,5}=a^1_{4,5}=a^1_{3,6}=a^1_{4,6}=a^1_{5,6}=a^1_{3,7}=a^1_{4,7}=a^1_{5,7}=a^1_{6,7}=0\\
a^2_{3,4} = a^2_{3,5}=a^2_{4,5}=a^2_{3,6}=a^2_{4,6}=a^2_{5,6}=a^2_{3,7}=a^2_{4,7}=a^2_{5,7}=a^2_{6,7}=0\\
\end{eqnarray*}
and, using $a^1_{i,j}$, $3 \leq i < j \leq 7$, $a^2_{3,4}, a^2_{3,5}, a^2_{3,6}, a^2_{3,7}$ we find $$b_{2,4},b_{3,4}, b_{2,5}, b_{3,5},  b_{4,5}, b_{2,6}, b_{3,6}, b_{4,6}, b_{5,6}, b_{2,7}, b_{3,7}, b_{4,7}, b_{5,7}, b_{6,7}$$ depending of $\alpha_1$, $\gamma_1$, $\beta_2$ and $\gamma_2$.

We replace them in the remaining equations and use Groebner basis in order to write the ideal 
$$
\langle a^2_{4,5}, a^2_{4,6}, a^2_{4,7}, a^2_{5,6}, a^2_{5,7}, a^2_{6,7}  \rangle = \langle \gamma_2, \alpha_1 \beta_2 \rangle
$$
but this implies that $\alpha_1 \beta_2 =0$, which is not possible.

If $b_{2,3}=0$ with a similar argument we arrive in $\lambda_1 = \lambda_2 =0$ and thus we conclude by lemma \ref{(a1,0,0,1), (0,b2,0,1)}.\\
\end{proof}

We finish by giving the idea for proving the second part of Theorem \ref{Principal}. Assume then that $C^2=n \geq 2$ and $(U,C)$ admits two fibrations transverse to $C$. We can prove that the tangency locus is not $C$. Now, blowing up the curve $C$ at $n$ different points and looking to the tangency locus between the Riccati foliations obtained in the neighborhood of the transform of $C$, which has zero self-intersection, we see that $C^2=2$ and the fibrations have empty tangency locus. We just take first integrals $g,g':U\to C\simeq\mathbb P^1$ defined by 
the two fibrations and remark that the map $(g,g'):U\to \Pp^1 \times \Pp^1$ is a biholomorphism over a neighborhood of the diagonal curve.

All the results of this note will be detailled and completed in \cite{Nous}.


\begin{thebibliography}{99}
\frenchspacing

\bibitem{BPV}
{\sc W. Barth, C. Peters, A. Van de Ven },
\emph{Compact complex surfaces.}
Springer-Verlag (1984).

\bibitem{Nous}
{\sc M. Falla Luza and F. Loray}
\emph{Projective connections and neighborhoods of rational lines.}
In preparation.

\bibitem{Grauert}
{\sc H. Grauert}
\emph{\"Uber Modifikationen und exzeptionelle analytische Mengen.}
Math. Ann. 146 (1962) 331-368. 

\bibitem{Hur}
{\sc J. C. Hurtubise and N. Kamran },
\emph{Projective connections, double fibrations, and formal neighborhoods of lines.}
Math. Ann. 292 (1992) 383-409.

\bibitem{LB}
{\sc S. R. LeBrun Jr.},
\emph{Spaces of complex geodesics and related structures.}
Phd thesis.

\bibitem{Mish}
{\sc M. B. Mishustin},
\emph{Neighborhoods of the Riemann sphere in complex surfaces.}
Funct. Anal. Appl. 27 (1993) 176-185.

\bibitem{Savelev}
{\sc V. I. Savel'ev},
\emph{Zero-type imbedding of a sphere into complex surfaces.}
Vestnik Moskov. Univ. Ser. I Mat. Mekh. 1982, no. 4, 28-32, 85. 

\bibitem{Wo}
{W. Kry\'nski},
\emph{Webs and projective structures on a plane.}
Differential Geometry and its Applications, 7 (2014) 133-140.


\end{thebibliography}
\end{document}